\documentclass[letterpaper,10pt,conference]{ieeeconf}

\IEEEoverridecommandlockouts
\newtheorem{theorem}{Theorem}
\newtheorem{algorithm}{Algorithm}

\newtheorem{corollary}{Corollary}

\usepackage{xurl}
\usepackage{csquotes}
\usepackage{dirtytalk}
\usepackage{algorithm}  

\usepackage{algpseudocode}
\usepackage{amssymb}
\usepackage{graphicx}
\usepackage[cmex10]{amsmath}
\usepackage{array}
\usepackage{mdwtab}
\usepackage{fixltx2e}
\usepackage{graphics}
\usepackage{epsfig}
\usepackage{times}
\usepackage{amsmath}
\usepackage{amssymb}
\usepackage{url}
\usepackage{lscape}
\usepackage{color}
\usepackage{epsfig}
\usepackage{cite}
\usepackage{balance}
\usepackage{dblfloatfix}
\usepackage{hyperref}
\usepackage{caption}
\usepackage{subcaption}

\title{\LARGE \bf A Block-Alternating Iterative Approach \\ 
for a Class of Non-Convex Optimization Problems}

\author{Anran Li, John P. Swensen, and Mehdi Hosseinzadeh,~\IEEEmembership{Senior Member,~IEEE}
\thanks{This research was supported by the National Science Foundation under award number CNS-2502856.}
\thanks{The authors are with the School of Mechanical and Materials Engineering, Washington State University, Pullman, WA 99164, USA (email: anran.li@wsu.edu; john.swensen@wsu.edu; mehdi.hosseinzadeh@wsu.edu).}
}

\begin{document}

\maketitle
\thispagestyle{empty}
\pagestyle{empty}

\begin{abstract}
Constrained non-convex optimization problems frequently arise in control applications. Solving such problems is inherently challenging, as existing methods often converge to suboptimal local minima or incur prohibitive computational costs. To address this challenge, this paper proposes a novel block-alternating iterative method that decomposes the original problem into variable-specific subproblems, which are solved iteratively. Under the assumption that the problem is convex with respect to each decision variable, the proposed approach reformulates the original problem into a sequence of convex subproblems. Theoretical results are established regarding the convergence and optimality of the method. In addition, a numerical example and a real-world control engineering application are presented to demonstrate its effectiveness. Finally, this paper introduces a ready-to-use Python platform that implements the proposed method, together with existing algorithms, to facilitate comparison and adoption.

\end{abstract}

\section{Introduction}
Consider the following optimization problem:
\begin{equation}\label{MainProblem1}
x_i^\ast,~\forall i=\left\{
\begin{array}{ll}
     & \arg\;\min f(x_1,\cdots,x_n) \\
    \text{s.t.} & g_j(x_1,\cdots,x_n)\leq0,~j=1,\cdots,c
\end{array}
\right.,
\end{equation}
where $c$ denotes the number of constraints. Here, the function $f(\cdot)$ and the set $\{x_1,\cdots,x_n:\; g_j(x_1,\cdots,x_n)\leq0,~\forall j\}$ are non-convex with respect to the joint variable $(x_1,\cdots,x_n)$, but convex with respect to each individual variable $x_i,~\forall i$. Assume that the optimization problem \eqref{MainProblem1} admits a unique global optimum and the feasible set $\{x_1,\cdots,x_n:g_j(x_1,\cdots,x_n)\leq0,~\forall j\}$ is connected.


Problem \eqref{MainProblem1} appears in a variety of applications, ranging from signal processing and communications \cite{zeng2020coordinate} to control and robotics \cite{shah2021rapid,parra2021rotation}, as well as in learning \cite{yuan2023coordinate}. In particular, the one-step-ahead predictive control framework presented in our prior work \cite{Our-one-step-ahead,Li2025ACC} determines the control input by solving an optimization problem of the form \eqref{MainProblem1}.

In many applications, the optimization problem \eqref{MainProblem1} is highly complex and often exceeds the capabilities of conventional gradient-based solvers \cite{nocedal2014interior,kungurtsev2014sequential,Boyd2004}. Metaheuristic algorithms, such as genetic algorithms and particle swarm optimization, are generally impractical due to their high computational cost \cite{alhijawi2024genetic,shami2023velocity} and their tendency to converge to local, rather than global, optima \cite{lai2019adaptive,diamond2018general,xu2020second}. Neural optimization machines \cite{Chen2024} can be applied in specific cases; however, they remain computationally intensive and challenging to implement, particularly for medium- to large-scale problems.

A promising alternative is provided by alternating methods, including coordinate descent \cite{CD-Original-SJ,gu2020coordinate,rodomanov2020randomized} and the alternating-direction method of multipliers \cite{glowinski2014alternating}. These methods exploit problem structure by optimizing over one subset of variables (or `block') at a time while keeping the others fixed. In doing so, they convert the original non-convex problem into a sequence of simpler subproblems that are often convex and can be efficiently solved using standard gradient-based methods. However, alternating optimization methods are typically most effective for unconstrained problems, since optimizing a single coordinate can easily violate constraints in the general case. To address this limitation, projected coordinate descent and related techniques \cite{shi2016primer} can incorporate simple constraints; however, their applicability is generally restricted to specific problem classes.

This paper addresses the aforementioned challenges by developing a novel alternating optimization method applicable to constrained problems of the form \eqref{MainProblem1}. The convergence properties of the proposed method are rigorously established. Its effectiveness is demonstrated through a representative numerical optimization problem. Moreover, the method has been integrated into the one-step-ahead predictive control framework described in \cite{Our-one-step-ahead} for temperature regulation. To facilitate broader adoption and enable other researchers to easily apply the method, a Python platform implementing the algorithm has been developed and shared; see \cite{Optimization-GUI}.

The rest of the paper is organized as follows. Section \ref{sec:Block-Alternating Iterative Opt} introduces the proposed block-alternating iterative optimization method and establishes its theoretical properties. Section \ref{sec:Experimental Study} presents a numerical example and a control engineering application to demonstrate the effectiveness of the method. Section \ref{sec:Python Platform} describes the developed Python platform. Finally, Section \ref{sec:Conclusion} concludes the paper.

\section{Block-Alternating Iterative Optimization}\label{sec:Block-Alternating Iterative Opt}


This subsection presents the proposed method for solving optimization problems of the form \eqref{MainProblem1}.

At the initial iteration, we consider the vector $(x_1^{(0)},x_2^{(0)},x_3^{(0)},\cdots,x_n^{(0)})$ as the initial guess for the optimal solution of the optimization problem \eqref{MainProblem1}. This initial point can be generated randomly from the feasible set, defined as $\{x_1,\cdots,x_n:\; g_j(x_1,\cdots,x_n)\leq0,~\forall j\}$; see Subsection \ref{sec:Sampling} for more details.

In the block-alternating iterative optimization, each iteration optimizes a single component of the decision vector while keeping the remaining components fixed. Specifically, in the first stage, we solve the subproblem
\begin{equation}
\tilde{x}_{1,1}^\ast=\left\{
\begin{array}{ll}
     & \arg\;\min\limits_{x_1} f(x_1,x_2^{(0)},x_3^{(0)},\cdots,x_n^{(0)}) \\
    \text{s.t.} & g_j(x_1,x_2^{(0)},x_3^{(0)},\cdots,x_n^{(0)})\leq0,~\forall j
\end{array}
\right.,
\end{equation}
and update the solution as $(x_1^{(1)},x_2^{(1)},x_3^{(1)},\cdots,x_n^{(1)}):=(\tilde{x}_{1,1}^\ast, x_2^{(0)},x_3^{(0)},\cdots,x_n^{(0)})$. Next, we optimize with respect to the second decision variable, yielding
\begin{equation}
\tilde{x}_{2,2}^\ast=\left\{
\begin{array}{ll}
     & \arg\;\min\limits_{x_2} f(x_1^{(1)},x_2,x_3^{(1)},\cdots,x_n^{(1)}) \\
    \text{s.t.} & g_j(x_1^{(1)},x_2,x_3^{(1)},\cdots,x_n^{(1)})\leq0,~\forall j
\end{array}
\right.,
\end{equation}
and update accordingly as $(x_1^{(2)},x_2^{(2)},x_3^{(2)},\cdots,x_n^{(2)}):=(x_1^{(1)}, \tilde{x}_{2,2}^\ast,x_3^{(1)},\cdots,x_n^{(1)})$. 

This process is repeated sequentially over all decision variables, i.e., from $x_1$ through $x_n$, after which the cycle is restarted. Specifically, at iteration $k$ and for the $i$-th decision variable, the following subproblem is solved:
\begin{equation}\label{eq:Iterationk}
\tilde{x}_{i,k}^\ast=\left\{
\begin{array}{ll}
     & \arg\;\min\limits_{x_i} f(x_1^{(k-1)},\cdots,x_{i-1}^{(k-1)},x_i,\\
     & ~~~~~~~~~~x_{i+1}^{(k-1)},\cdots,x_n^{(k-1)}) \\
    \text{s.t.} & g_j(x_1^{(k-1)},\cdots,x_{i-1}^{(k-1)},x_i,\\
     & ~~~~~~~~~~x_{i+1}^{(k-1)},\cdots,x_n^{(k-1)})\leq0,~\forall j
\end{array}
\right.,
\end{equation}
and once $\tilde{x}_{i,k}^\ast$ is obtained, the solution vector is updated as $(x_1^{(k)},\cdots,x_{i-1}^{(k)},x_i^{(k)},x_{i+1}^{(k)},\cdots,x_n^{(k)})
:= (x_1^{(k-1)},\cdots,x_{i-1}^{(k-1)},\tilde{x}_{i,k}^\ast,x_{i+1}^{(k-1)},\cdots,x_n^{(k-1)})$.

The iterative scheme continues until convergence, typically determined when successive iterates remain unchanged or satisfy a predefined stopping criterion. Algorithm \ref{alg:algorithm} provides the pseudocode of the proposed optimization method, where $rem(k,n)$ denotes the remainder of the division $k/n$.



\begin{algorithm}[!t]
    \caption{Block-Alternating Iterative Optimization}
    \label{alg:algorithm}
    \begin{algorithmic}[1]
        \State \textbf{Initialize:} Select a feasible initial point $(x_1^{(0)}, x_2^{(0)},x_3^{(0)},\cdots,\,x_n^{(0)})$ 
        \For{$k=1,2,\cdots$} 
            \State  $i\leftarrow rem(k,n)$
                \If {$i=0$} 
                    \State $i\leftarrow n$
                \EndIf
            \State $\tilde{x}_{i,k}^\ast \leftarrow$ Solve the Optimization Problem \eqref{eq:Iterationk}
            \State {\scriptsize $(x_1^{(k)},\cdots,x_n^{(k)}) \leftarrow(x_1^{(k-1)},\cdots,\tilde{x}_{i,k}^\ast,\cdots,x_n^{(k-1)})$}           
        \If {stopping criterion is satisfied}
        \State \Return $(x_1^{(k)},\cdots,x_n^{(k)})$
        \EndIf
        \EndFor
    \end{algorithmic}   
\end{algorithm}

\subsection{Theoretical Properties}\label{sec:Theoretical-Preperties}
This subsection studies the theoretical properties of the proposed method. We begin with the following corollary.

\begin{corollary}\label{cor:corollary1}
The iterative structure of the method implies that, if the initial point $(x_1^{(0)},\cdots,x_n^{(0)})$ is feasible, then the updated solution $(x_1^{(k)},\cdots,x_n^{(k)})$ remains feasible for all $k \geq 0$. Moreover, the cost function $f(\cdot)$ satisfies the inequality $f(x_1^{(k)}, \dots, x_n^{(k)}) \leq 
f(x_1^{(k-1)}, \dots, x_n^{(k-1)})$ for all $k\geq 0$,
indicating that the method generates a non-increasing sequence of cost values.
\end{corollary}

The next theorem investigates the asymptotic convergence properties of the proposed optimization method.

\begin{theorem}\label{theorem:optimal}
Consider the optimization problem in \eqref{MainProblem1}, subject to the conditions specified thereafter. Suppose that the block-alternating iterative optimization method starts from a feasible initial point $(x_1^{(0)},\cdots,\,x_n^{(0)})$. Then, if cost function $f(\cdot)$ is strongly convex with respect to each individual variable $x_i$ when all other variables are fixed, the tuple $(x_1^{(k)},\cdots,x_n^{(k)})$ generated by the method converges to the optimal solution $(x_1^{\ast},\cdots,x_n^{\ast})$ as $k\rightarrow\infty$.
\end{theorem}

\begin{proof}
Corollary \ref{cor:corollary1} established that the cost function $f(\cdot)$ is non-increasing along the iterations. Let $\underline{J}$ denote its limiting value as $k\rightarrow\infty$, i.e.,
\begin{align}\label{eq:Converge1}
\lim_{k\rightarrow\infty}{f(x_1^{(k)}, \dots,\tilde{x}_{i,k}^\ast,\cdots,x_n^{(k)})}&=\underline{J},
\end{align}
where $\tilde{x}_{i,k}^\ast$ is defined in \eqref{eq:Iterationk}. Under the assumption that the cost function $f(\cdot)$ is strongly convex with respect to each individual variable $x_i$ and given that the feasible set is convex (see the discussion after Equation \eqref{MainProblem1}), the minimizer $\tilde{x}_{i,k}^\ast$ is unique. Consequently, it follows from \eqref{eq:Converge1} that there exists a unique tuple $(x_1^{\dag},\cdots,x_n^{\dag})$ such that $(x_1^{(k)},\cdots,x_n^{(k)})\rightarrow\big(x_1^{\dag},\cdots,x_n^{\dag}\big)$ as $k\rightarrow\infty$ (see Appendices \ref{Appendix:Uniqueness} and \ref{Appendix:ChatGPT}). Thus, to complete the proof, it is sufficient to show that $(x_1^{\ast},\cdots,x_n^{\ast})=(x_1^{\dag},\cdots,x_n^{\dag})$, which we will demonstrate by contradiction in the following.

Suppose that $(x_1^{\ast},\cdots,x_n^{\ast})\neq(x_1^{\dag},\cdots,x_n^{\dag})$. On one hand, we have\footnote{For brevity, we use $f(\mathbf{x}^\ast)$ to denote $f(x_1^{\ast},\cdots,x_n^{\ast})$.}:
\begin{align}\label{eq:theorem3-optimal1_n}
\big(\nabla f(\mathbf{x}^\ast)\big)^\top \begin{bmatrix}
x_1^\ast-x_1^\dag\\
\vdots\\
x_n^\ast-x_n^\dag
\end{bmatrix}=\sum_{i=1}^n[\nabla f(\mathbf{x}^\ast)]_i(x_i^\ast-x_i^\dag),
\end{align}
where $[\nabla f(\cdot)]_i$ indicates the $i$th element of $\nabla f(\cdot)$.

On the other hand, from the convexity of the cost function $f(\cdot)$ with respect to $x_i$, it follows that:
\begin{align}\label{eq:theorem3-optimal1}
&\big(\nabla f(\mathbf{x}^\ast)\big)^\top \begin{bmatrix}
x_1^\ast-x_1^\ast\\
\vdots\\
x_i^\ast-x_i^\dag\\
\vdots\\
x_n^\ast-x_n^\ast
\end{bmatrix}=[\nabla f(\mathbf{x}^\ast)]_i(x_i^\ast-x_i^\dag)\leq\Delta f_i,
\end{align}
where $\Delta f_i:=f(x_1^{\ast},\cdots,x_i^\ast,\cdots,x_n^{\ast})-f(x_1^{\ast},\cdots,x_i^\dag,\cdots,x_n^{\ast})$ and assuming that $(x_1^{\ast},\cdots,x_i^\dag,\cdots,x_n^{\ast})$ is a feasible point, the optimality of $(x_1^{\ast},\cdots,x_i^\ast,\cdots,x_n^{\ast})$ implies that $\Delta f_i<0$.

By combining \eqref{eq:theorem3-optimal1_n} and \eqref{eq:theorem3-optimal1}, we obtain:
\begin{align}\label{eq:Optimalitynew1}
\sum_{i=1}^n[\nabla f(\mathbf{x}^\ast)]_i(x_i^\ast-x_i^\dag)\leq \Delta f,
\end{align}
where $\Delta f:=\sum_{i=1}^n\Delta f_i$, which is strictly negative. Thus, if $(x_1^{\ast},\cdots,x_n^{\ast})\neq(x_1^{\dag},\cdots,x_n^{\dag})$, it follows from \eqref{eq:Optimalitynew1} that there exists at least one index $j\in\{1,\cdots,n\}$ such that\footnote{Suppose that $\sum_{i=1}^nz_i\leq \psi$, where $z_i\in\mathbb{R}$. Assume $z_i>\psi/n$ for all $i$. Adding these inequalities gives $\sum_{i=1}^nz_i>\psi$, which is a contradiction. Thus, it is impossible for all $z_i>\psi/n$ to hold simultaneously. Hence, there exists at least one index $j$ such that $z_j\leq\psi/n$.}:
\begin{align}\label{eq:Optimalitynew2}
[\nabla f(\mathbf{x}^\ast)]_j (x_j^\ast - x_j^\dag) \leq \frac{\Delta f}{n}.
\end{align}





From the Taylor expansion of the cost function $f(\cdot)$, we have:
\begin{align}\label{theorem3-opt2}
&f\left(x_1^\ast, \cdots, x_j^\ast+\epsilon(x_j^\ast -x_j^\dag),\cdots,x_n^\ast\right) = f(\mathbf{x}^\ast)\nonumber \\
&+\epsilon[\nabla f(\mathbf{x}^\ast)]_j(x_j^*-x_j^\dag) + \beta,
\end{align}
for some $\epsilon>0$, where $\beta$ accounts for the higher-order terms in the Taylor expansion, i.e., terms involving higher powers of $\epsilon$ multiplied by the corresponding higher-order derivatives of $f(\cdot)$. According to \eqref{eq:Optimalitynew2}, it follows from \eqref{theorem3-opt2} that:
\begin{align}\label{theorem3-opt3}
&f\left(x_1^\ast, \cdots, x_j^\ast+\epsilon(x_j^\ast -x_j^\dag),\cdots,x_n^\ast\right) \leq f(\mathbf{x}^\ast)\nonumber\\
&+\epsilon\frac{\Delta f}{n}+ \beta.
\end{align}

Conversely, from the optimality of the solution $(x_1^{\ast},\cdots,x_j^\ast,\cdots,x_n^{\ast})$ and assuming that $\left(x_1^\ast, \cdots, x_j^\ast+\epsilon(x_j^\ast -x_j^\dag),\cdots,x_n^\ast\right)$ is a feasible point, we have:
\begin{align}\label{theorem3-opt4}
f(\mathbf{x}^\ast)\leq f\left(x_1^\ast, \cdots, x_j^\ast+\epsilon(x_j^\ast -x_j^\dag),\cdots,x_n^\ast\right).
\end{align}

Thus, from \eqref{theorem3-opt3} and \eqref{theorem3-opt4}, it implies that $\epsilon\frac{\Delta f}{n}+ \beta\geq0$. However, since $\Delta f<0$, $\beta$ contains higher powers of $\epsilon$, and $\epsilon$ is an arbitrary constant that can be chosen such that $\epsilon\frac{\Delta f}{n}+\beta<0$, a contradiction is reached. Thus, it must be that $(x_1^{\dag},\cdots,x_n^{\dag})=(x_1^{*},\cdots,x_n^{*})$, thereby completing the proof.
\end{proof}

\subsection{Non-Strongly Convex Cost Functions}\label{sec:Sampling}
When the cost function $f(\cdot)$ is convex\textemdash but not strongly convex\textemdash with respect to the individual variables $x_i$, the solution to the optimization problem \eqref{eq:Iterationk} is not guaranteed to be unique. Consequently, the proposed block-alternating iterative optimization method may converge to a local minimum. To enhance the likelihood of reaching the global minimum, and drawing inspiration from the literature on genetic algorithms (GA) \cite{alhijawi2024genetic} and particle swarm optimization (PSO) \cite{kennedy1995PSO}, we propose initializing Algorithm~\ref{alg:algorithm} from multiple starting points. These independent runs can be executed in parallel, which helps reduce the overall computational cost.

Existing methods include Monte Carlo sampling \cite{hastings1970MC_sampling}, Latin Hypercube Sampling (LHS) \cite{mckay2000LHS}, {grid sampling \cite{WANG20121sampling}}, and {knowledge-based sampling \cite{WANG20121sampling}}. While Monte Carlo and grid sampling are straightforward to implement, they often become impractical for high-dimensional optimization problems. In contrast, LHS and knowledge-based sampling provide greater diversity, but their implementation is generally more computationally demanding \cite{shan2010survey-sampling}.

To overcome the limitations of existing sampling methods, we propose a hybrid strategy that integrates grid sampling with LHS. Initially, we define the number of points, $N$, to initialize the algorithm. The feasible domain is discretized into equal intervals for each variable, forming a candidate set. From this set\footnote{We use $round(N/2)$ to denote the nearest integer to $N/2$.}, $round(N/2)$ points are selected based on their low cost function values. Unlike conventional grid sampling, the discretization step can be relatively large, reducing computational burden. The remaining $N-round(N/2)$ points are generated using LHS. This hybrid approach ensures that a portion of the initial points is concentrated in promising low-cost regions, while the LHS-generated points maintain diversity and mitigate sampling bias. 

Figure~\ref{fig:sampling} illustrates the grid sampling, LHS, and the proposed hybrid sampling method applied to a representative objective function. Importantly, when the objective function features a global minimum within a sharp and narrow basin, combining these two methods increases the likelihood of sampling points near this region, thereby outperforming pure grid sampling or LHS in such scenarios.

\begin{figure}[!t]
    \centering
    \includegraphics[width=\columnwidth]{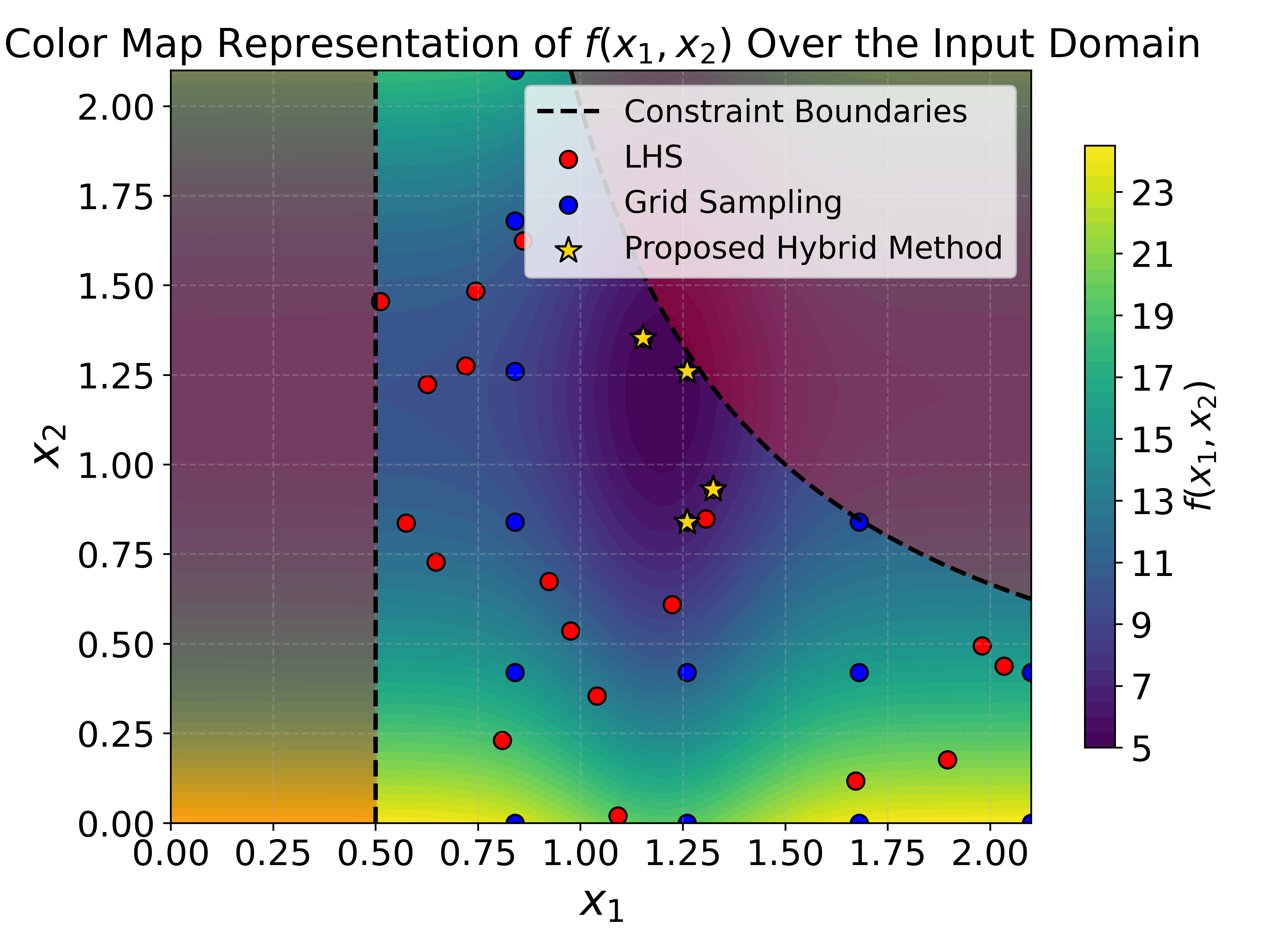}
    \caption{Geometric illustration of grid sampling, LHS, and the proposed hybrid sampling method for selecting $N=4$ initial points. {The cost function is $f(x_1,x_2)=-5e^{\frac{-(x_1-1.2)^2}{0.08}}+0.02(x_1-1.2)^4+0.02(x_1-1.2)^2+10(x_2-1.2)^2+10$, and the constraints are $x_1\geq0.5$ and $x_2\leq\frac{1}{x_1-0.5}$.}}
    \label{fig:sampling}
\end{figure}

\section{Experimental Studies}\label{sec:Experimental Study}
This section evaluates the effectiveness of the proposed optimization method in solving problems of the form \eqref{MainProblem1}. We begin with a numerical example to compare the proposed method against existing algorithms. Then, we demonstrate its application in a temperature control problem.

\subsection{Numerical Example}
Consider the following optimization problem:
\begin{subequations}\label{eq:eg.1}
\begin{align}
\min\;\;(x_1 - x_2)^2 + (\frac{1}{x_2}+2)^2+0.5x_3^2
\end{align} subject to 
\begin{align}
&0\leq x_1\leq3\\
&2\leq x_2\leq10\\
& -1\leq x_3\leq1\\
&x_1+x_2\leq 5\\
&x_1 x_3 \geq 2
\end{align}   
\end{subequations}
where the cost function and the constraints satisfy the conditions outlined in \eqref{MainProblem1}. We initialize the algorithm with 32 starting points, selected using the hybrid sampling strategy described in Subsection \ref{sec:Sampling}. As shown in Figure \ref{fig:cost plot}, the block-alternating iterative optimization method converges to the optimal solution with 6 iterations.


For comparison purposes, we evaluated the proposed method against PSO and GA approaches, using the same set of initial points to solve the optimization problem defined in \eqref{eq:eg.1}. Both the PSO and GA methods were run for a maximum of 200 iterations. All algorithms were executed in the same computing environment: a 12th Gen Intel\textsuperscript{\textregistered} Core\textsuperscript{\texttrademark} i5-1230U processor with 8 GB of RAM. The implementations were developed using \texttt{Python 3.10}.

Comparison results are presented in Figure \ref{fig:compare GA PSO}. The top figure shows the number of initial points required by each method to reach the optimal solution, while the bottom figure illustrates the corresponding computation times. As shown in the top figure, the proposed block-alternating iterative optimization method achieves convergence with fewer than 8 initial points, compared to approximately 8--16 for PSO and 256--512 for GA. The bottom figure demonstrates that while computation time increases with the number of initial points for all methods, the proposed block-alternating iterative optimization consistently requires the least amount of time.





\begin{figure}
    \centering
    \includegraphics[width=0.8\linewidth]{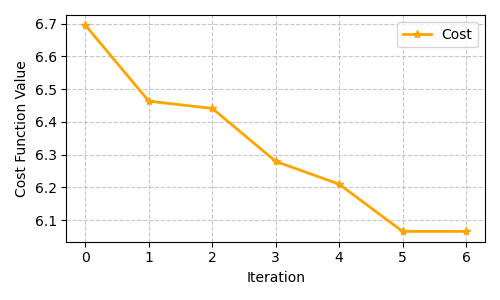}
    \caption{Convergence of the cost function over iterations with the block-alternating iterative optimization method. }
    \label{fig:cost plot}
\end{figure}

\begin{figure}
    \centering
    \includegraphics[width=1\linewidth]{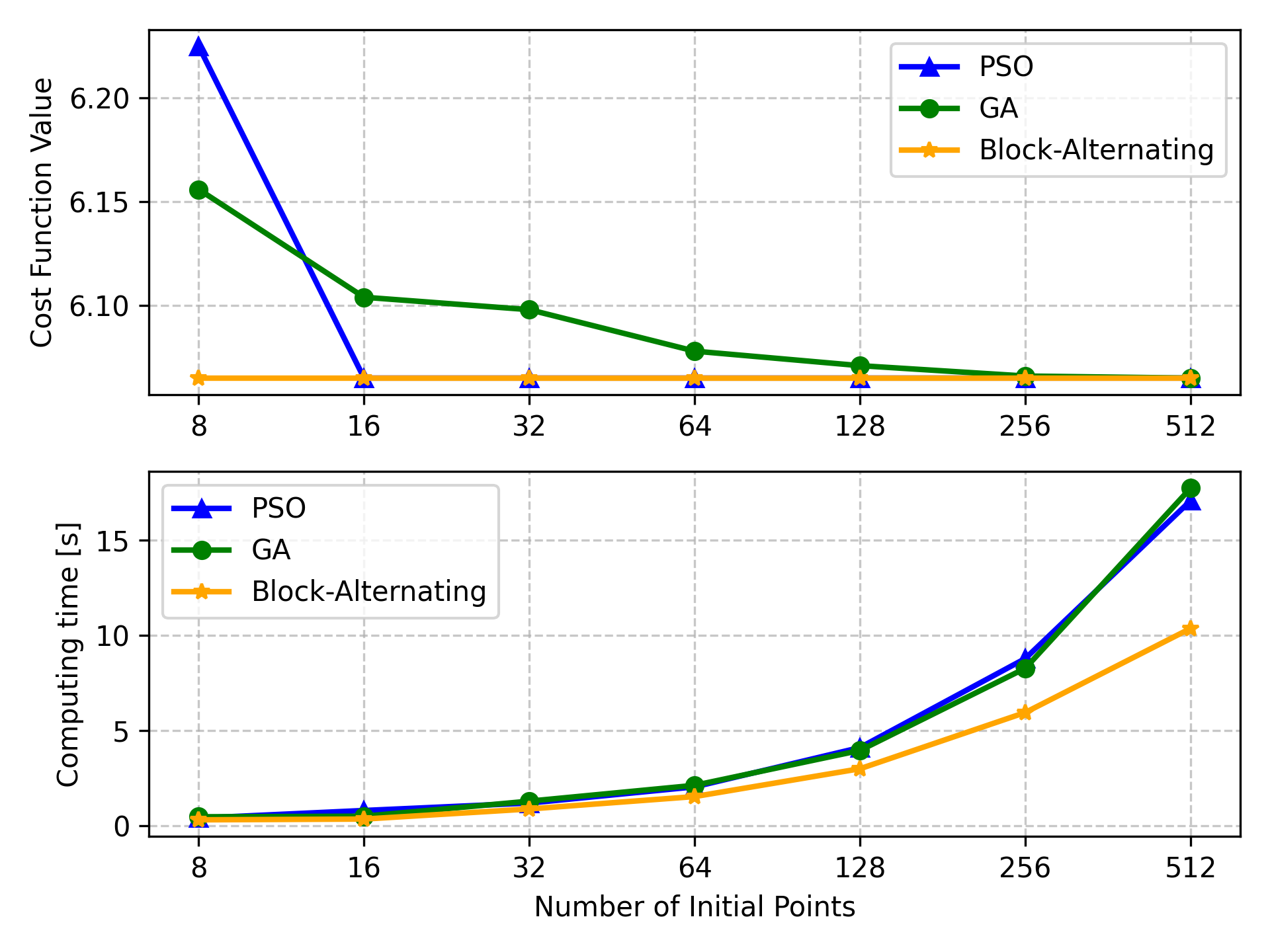}
    \caption{Comparison of optimization performance and computation time for the proposed block-alternating iterative method, PSO, and GA.}
    \label{fig:compare GA PSO}
\end{figure}


\subsection{Thermal Control System}
This subsection applies the proposed block-alternating iterative method to regulate the temperature of the system depicted in Fig. \ref{fig:heatsystem}, using the one-step-ahead predictive control approach described in \cite{Our-one-step-ahead,Li2025ACC}.


Utilizing the energy balance method, the dynamical model of this system is \cite{Oliveira2023}:
\begin{align}\label{eq:Heater}
mc_p\frac{dT}{dt}=UA(T_{amb.}-T)+\varepsilon\sigma A\big(T_{amb.}^4-T^4)+\alpha Q, 
\end{align}
where $m=4$ [g] is the mass, $c_p=500$ [J/Kg$\cdot$K] is the heat capacity, $T$ is the temperature expressed in [K], $U=10$ [W/m$^2\cdot$K] is the heat transfer coefficient, $A=12\times10^{-4}$ [m$^2$] is the area, $T_{amb.}$  is the ambient temperature expressed in [K], $\varepsilon=0.9$ is the emissivity, $\sigma=5.67\times10^{-8}$ [W/m$^2\cdot$K$^2$] is the Stefan-Boltzmann constant, $Q$ is the percentage heater output, and $\alpha=0.01$ is a factor that relates heater output to power dissipated by the transistor. 
 
As shown in \cite{Oliveira2020}, the system described by \eqref{eq:Heater} can be approximated by a first-order linear system with an added delay. To identify this model, we employ the open-loop step response method proposed in \cite{Oliveira2020_2} and discretize the identified model by using a sampling period of 0.5 second. The resulting system is:
\begin{align}\label{eq:thermal-ss}
x(t+1)=&\left[\begin{array}{cc}
0 & -0.0005 \\
1 & -0.0965 
\end{array}\right] x(t)+\left[\begin{array}{c}
0.0004 \\
-0.00
\end{array}\right] u(t), \nonumber\\
y(t)=&\left[\begin{array}{ll}
0 & 1
\end{array}\right] x(t)+T_{amb.},
\end{align}
where $x_1$ represents an internal state, and $x_2$ corresponds to the temperature in Celsius. A Luenberger observer is designed with the gain {$L=[0.85~0.9]^\top$} to estimate $x_1$.

\begin{figure}[!t]
    \centering
    \includegraphics[width=5.5cm]{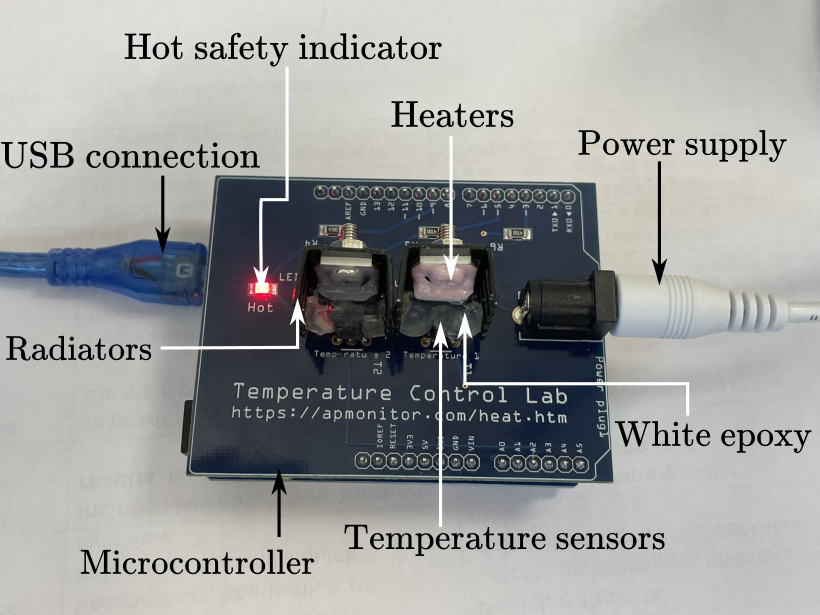}
    \caption{Experimental setup for temperature control.}
    \label{fig:heatsystem}
\end{figure}


The one-step-ahead predictive control framework \cite{Our-one-step-ahead,Li2025ACC} computes the control input $u(t)$, together with a Lyapunov matrix $P(t)$, at each time step $t$ by solving the following optimization problem:
\begin{subequations}\label{eq:OptimizationProblemMain}
\begin{align}\label{eq:CostFunction}
u^\ast(t),P^\ast(t)=\arg\,\min_{u,P}\,&\left\Vert x^+-\bar{x}_r\right\Vert^2_{Q_x}+\left\Vert u-\bar{u}_r\right\Vert^2_{Q_u}\nonumber\\
&+V\left(x(t),r,P\right),
\end{align}
subject to
\begin{align}
P\succ 0, \\
V(x^+,r,P)-V\left(x(t),r,P\right)<-\theta \Vert x(t)-\bar{x}_r\Vert_P,\label{eq:Constraint2}
\end{align}
\end{subequations}
where $V(x,r,P):=\left\Vert x-\bar{x}_r\right\Vert_{P}$ is the Lyapunov function with $P\succ0$ ($P\in\mathbb{R}^{2\times 2}$) being the Lyapunov matrix, $x^+$ is the next step based on the linearized model given in \eqref{eq:thermal-ss}, $Q_x=Q_x^{\top} \succeq 0$ ($Q_x \in \mathbb{R}^{2 \times 2}$) and $Q_u>0$ are weighting matrices, and $\theta>0$ is the contraction parameter.

The cost function and constraints of the optimization problem \eqref{eq:OptimizationProblemMain} satisfy the conditions specified in \eqref{MainProblem1}. Accordingly, we apply the proposed block-alternating iterative optimization method to this problem at each time step $t$, thereby regulating the system temperature.


Given $r=50$ Celsius and assuming that {$Q_x=\text{diag}\left\{0.1,10\right\}$ and $Q_u=0.001$}, Fig. \ref{fig:heater} presents the experimental results (estimated $\hat{x}_1(t)$, output $y(t)$, and control input $u(t)$), including a comparison with the iterative LQR reported in \cite{Prasad2014}. To guarantee that the optimization problem \eqref{eq:OptimizationProblemMain} can be solved within 0.5 seconds (i.e., the sampling period), the number of initial starting points in the proposed block-alternating iterative optimization method is set to 16.

As shown in Fig. \ref{fig:heater}, both methods successfully drive the system states to the desired equilibrium point. However, the one-step-ahead predictive control method, implemented via the block-alternating iterative approach, produces a smoother and less oscillatory control input compared to the LQR method. This improvement arises because the one-step-ahead predictive control relies on a single-step prediction, which substantially reduces the influence of linearization errors on closed-loop performance. Furthermore, Fig. \ref{fig:heater} demonstrates the effectiveness of the proposed block-alternating iterative optimization method in solving the non-convex optimization problem in \eqref{eq:OptimizationProblemMain}. 


\begin{figure}[!t]
    \centering
    \includegraphics[width=6.5cm]{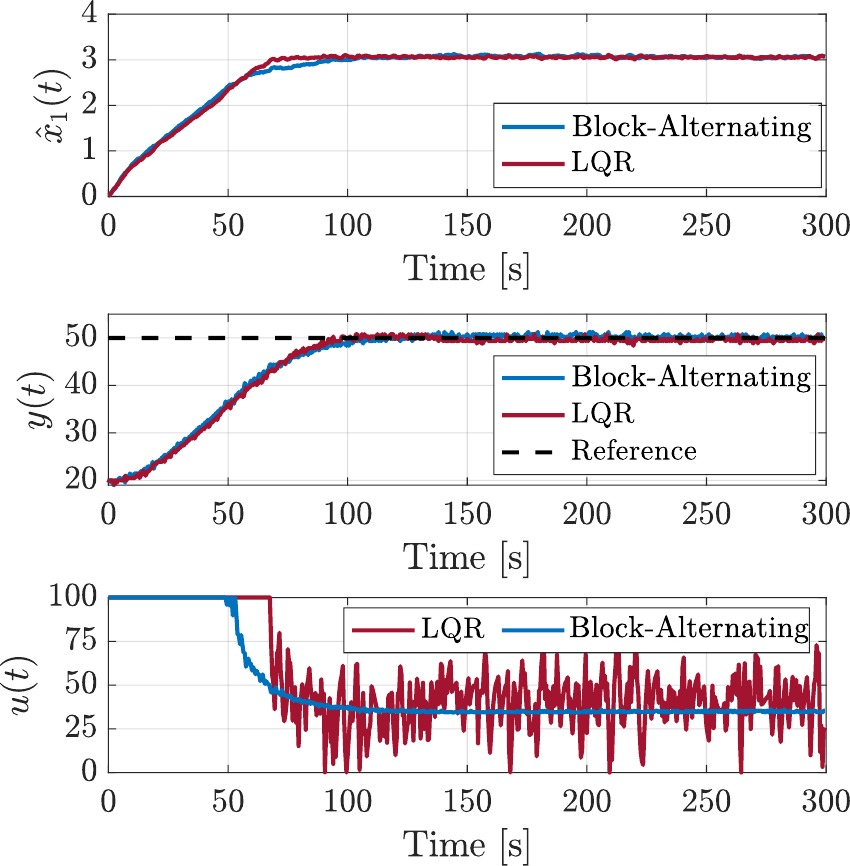}
    \caption{Experimental results comparing the temperature control performance of the one-step-ahead predictive control method, implemented via the block-alternating iterative approach, with that of the LQR method.}
    \label{fig:heater}
\end{figure}

\section{Python Platform}\label{sec:Python Platform}

To enable other researchers to apply the proposed optimization method in their studies and applications, this section introduces a Python platform that is readily accessible. The platform is publicly available at {\cite{Optimization-GUI}}.

The layout of the platform is shown in Fig. \ref{fig:platform}. First, the user specifies the number of decision variables and the initial points to initialize the algorithm. If the number of initial points is one, the user may manually set its value. Next, the user defines the search domain by setting the minimum and maximum allowable values for each decision variable, along with the cost function and constraints. To facilitate ease of use, the input format follows standard Python mathematical syntax, and a syntax-error window is automatically triggered if any invalid expressions are detected.


To enhance flexibility and enable method comparison, the platform allows the user to select the optimization algorithm. The available options include the proposed block-alternating iterative method, as well as Genetic Algorithm (GA) and Particle Swarm Optimization (PSO). 



The platform also includes a results display window that presents the solution outcomes. If the optimization problem is infeasible or ill-posed (e.g., when the objective is non-convex in all variables), the platform provides diagnostic feedback. When the problem is successfully solved, the platform visualizes the cost function values and iteration history at the bottom of the interface. Users can also export these plots using the `Export plot' button, which saves the data points in JSON format for further analysis.

\begin{figure}
    \centering
    \includegraphics[width=\columnwidth]{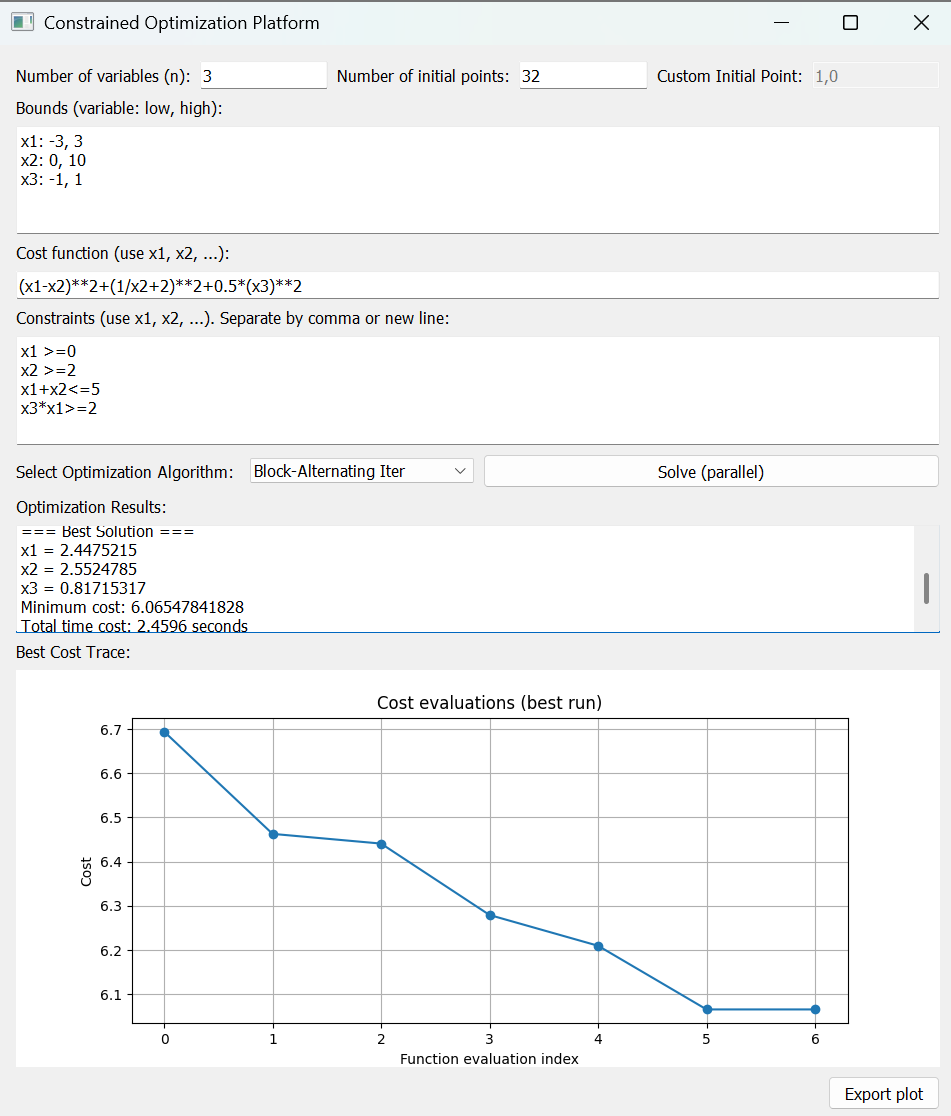}
    \caption{Layout of the developed Python platform, illustrating the setup of the optimization problem in \eqref{eq:eg.1}. The platform is publicly available at \cite{Optimization-GUI}.}
    \label{fig:platform}
\end{figure}

\section{Conclusion}\label{sec:Conclusion}
Constrained non-convex optimization problems frequently arise in engineering and scientific applications. Conventional solvers often incur high computational costs and are prone to convergence to local minima. To address these challenges, this paper developed a block-alternating iterative approach for solving such optimization problems. The method decomposes the original problem into variable-specific subproblems, which are solved iteratively until convergence to an optimal solution. Theoretical properties of the proposed method were rigorously established. The method was validated and evaluated through a numerical example and a control engineering application. Moreover, a Python platform providing end-to-end functionality was developed and made publicly available to facilitate its use by other researchers.



\bibliography{references} 
\bibliographystyle{ieeeconf}

\appendices

\section{Uniqueness of the Tuple $(x_1^{\dag},\cdots,x_n^{\dag})$}\label{Appendix:Uniqueness}

Suppose that the cost function $f(\cdot)$ is strongly convex and the feasible set $\{x_1,\cdots,x_n:\; g_j(x_1,\cdots,x_n)\leq0,~\forall j\}$ is convex, both with respect to each individual variable $x_i$ when all other variables are fixed. Let $\underline{J}$ denote the limiting value of the cost function $f(\cdot)$ as $k \to \infty$ (see Corollary \ref{cor:corollary1}). Suppose $x_i^\dag$ is the solution of the optimization problem \eqref{eq:Iterationk} in the limit $k \to \infty$. This implies that $f(x_1^{(k-1)}, \dots, x_i^\dag, x_{i+1}^{(k-1)}, \dots, x_n^{(k-1)}) = \underline{J}$.

Now consider iteration $k+1$, where we optimize over the decision variable $x_{i+1}$, and let the corresponding optimal solution be $x_{i+1}^\dag$. The uniqueness claim implies that $x_{i+1}^\dag=x_{i+1}^{(k-1)}$. Suppose instead that this equality does not hold.

Since $x_{i+1}^{(k-1)}$ is a feasible solution for the optimization problem at iteration $k+1$, the optimality of $x_{i+1}^\dag$ ensures that either $f(x_1^{(k-1)}, \dots, x_i^\dag, x_{i+1}^\dag, \dots, x_n^{(k-1)}) < f(x_1^{(k-1)}, \dots, x_i^\dag, x_{i+1}^{(k-1)}, \dots, x_n^{(k-1)})$ (Condition I), or $f(x_1^{(k-1)}, \dots, x_i^\dag, x_{i+1}^\dag, \dots, x_n^{(k-1)}) = f(x_1^{(k-1)}, \dots, x_i^\dag, x_{i+1}^{(k-1)}, \dots, x_n^{(k-1)})$ (Condition II).

Condition I implies that $f(x_1^{(k-1)}, \dots, x_i^\dag, x_{i+1}^\dag, \dots, x_n^{(k-1)}) < \underline{J}$, which contradicts the convergence property established earlier. If Condition II holds, then the cost $\underline{J}$ is attained at two distinct points, contradicting the fact that a strongly convex cost function over a convex feasible set admits a unique minimizer.

Therefore, when $f(\cdot)$ is strongly convex in each variable $x_i$, the proposed algorithm converges to a unique tuple $(x_1^{\dag}, \dots, x_n^{\dag})$ as $k\rightarrow\infty$.


\section{Examination of Counterexamples Generated by ChatGPT and Copilot}\label{Appendix:ChatGPT}
Despite the discussion provided in Appendix \ref{Appendix:Uniqueness}, we attempted to use ChatGPT and Copilot to generate potential counterexamples to the uniqueness claim. While both platforms produced examples they claimed were valid counterexamples, upon closer examination we found that none of them satisfied the setting described after Equation \eqref{MainProblem1}.

The first two examples were the unconstrained optimization problems $\min\limits_{x_1,x_2}~(x_1 - x_2)^2$ and $\min\limits_{x_1,x_2}~ x_1^2+x_2^2 - 3 x_1 x_2+0.1(x_1^2+x_2^2)^2$, which indeed do not have a unique solution and therefore fall outside the scope of our setting. The third example was $\min\limits_{x_1,x_2}~(x_1 - 1)^2 + (x_2 + 1)^2$ subject to $x_1 x_2 + 0.5 \leq 0$ and $(x_1 - 1)(x_2 + 1) - 0.5 \leq 0$; in this case, the feasible set is not connected, which again violates the assumptions in our setting.  Another example was $\min\limits_{x_1,x_2}~x_1^2 + x_2^2$ subject to $x_1 = x_2$, but this introduces equality constraints, which are excluded from our formulation.

In summary, although ChatGPT and Copilot generated candidate counterexamples, none of them are valid under the assumptions of our problem setting. This further supports the uniqueness claim established in Appendix \ref{Appendix:Uniqueness}.


\end{document}